\documentclass[12pt]{article}
\usepackage{longtable}
\usepackage{textcomp}
\usepackage{amsmath}
\usepackage{amssymb}
\usepackage{mathrsfs}
\usepackage{amsfonts}
\usepackage{multirow}
\usepackage{cases}
\usepackage{graphicx}
\usepackage{tabularx}

\newtheorem{Theorem}{Theorem}[section]
\newtheorem{Definition}{Definition}[section]
\newtheorem{Proposition}{Proposition}[section]
\newtheorem{Lemma}{Lemma}[section]
\newtheorem{Example}{Example}[section]
\newtheorem{Corollary}{Corollary}[section]

\newcommand{\Section}[1]{
        \par
        \stepcounter{section}
        \settowidth{\hangindent}{\large\bf\thesection.~}
        \hangafter=1
        \bigskip\bigskip\noindent
        {\large\bf\hbox{\thesection.~}#1}\par
        \nopagebreak
        \medskip
}
\newcommand{\alglist}{
\begin{list}{Step 1}
{\setlength{\leftmargin}{1.1 in}\setlength{\labelwidth}{1.0 in}} }

\newenvironment{proof}{{\bf Proof.}\,}{\hfill$\hbox{\rule{5pt}{5pt}}$\\}

\begin{document}
\title{The Dominant Eigenvalue of an Essentially Nonnegative Tensor
}

\author{ Liping Zhang$^a$\thanks{Corresponding author. E-mail address: lzhang@math.stinghua.edu.cn} \quad\quad Liqun Qi$^b$\thanks{Email address: maqilq@polyu.edu.hk}\quad\quad Ziyan
Luo$^c$\thanks{E-mail address: starkeynature@hotmail.com}
\\
{\small $^a$ Department of Mathematical Sciences, Tsinghua University, Beijing, China}\\
{\small $^b$ Department of Applied Mathematics, The Hong Kong
Polytechnic University, Hong Kong}\\
{\small $^c$ Department of Mathematics, Beijing Jiaotong University,
Beijing, China}\\
 }

\date{}
\maketitle

{\bf Abstract.} It is well known that the dominant eigenvalue of a
real essentially nonnegative matrix is a convex function of its
diagonal entries. This convexity is of practical importance in
population biology, graph theory, demography, analytic hierarchy
process and so on. In this paper, the concept of essentially
nonnegativity is extended from matrices to higher order tensors, and
the convexity and log convexity of dominant eigenvalues for such a
class of tensors are established. Particularly, for any nonnegative
tensor, the spectral radius turns out to be the dominant eigenvalue
and hence possesses these convexities.  Finally, an algorithm is
given to calculate the dominant eigenvalue, and numerical results
are reported to show the effectiveness of the proposed algorithm.
\vspace{2mm}

{\bf Key words.} Essentially nonnegative tensor, dominant
eigenvalue, convex function, spectral radius, algorithm.

{\bf AMS subject classifications. }  74B99, 15A18, 15A69

\vspace{2mm}


\Section{Introduction}
Tensors are increasingly ubiquitous in
various areas of applied, computational, and industrial mathematics
and have wide applications in data analysis and mining, information
science, signal/image processing, and computational biology, etc;
see the workshop report \cite{direction09} and references therein. A
tensor can be regarded as a higher-order generalization of a matrix,
which takes the form
\begin{displaymath}
\mathcal{A}=\left(A_{i_1\cdots i_m}\right),\quad A_{i_1\cdots i_m}
\in R,\quad 1\le i_1,\ldots,i_m\le n.
\end{displaymath}
Such a multi-array $\mathcal{A}$ is said to be an {\it$m$-order
$n$-dimensional square real tensor} with $n^m$ entries $A_{i_1\cdots
i_m}$. In this regard, a vector is a first-order tensor and a matrix
is a second-order tensor. Tensors of order more than two are called
higher-order tensors.

Analogous with that of matrices, the theory of eigenvalues and
eigenvectors is one of the fundamental and essential components in
tensor analysis. $72$ references on eigenvalues of tensors can be
found in the bibliography \cite{s13}.  Wide range of practical
applications can be found the references there. Compared with that
of matrices, eigenvalue problems for higher-order tensors are
nonlinear due to their multilinear structure. Various types of
eigenvalues are defined for higher-order tensors in the setting of
multilinear algebra. For example, the eigenvalue, the
$H$-eigenvalue, the $E$-eigenvalue, the $Z$-eigenvalue, the
$N$-eigenvalue defined by Qi for even order symmetric tensors
\cite{s4}, the $l^p$ eigenvalues for general order symmetric
tensors, and the mode-$i$ eigenvalues for general square tensors
defined by Lim \cite{s2}, the $M$-eigenvalue for a partially
symmetric fourth-order tensor, defined by Qi, Dai and Han \cite{s9},
the $D$-eigenvalue for a fourth-order symmetric tensor and a
second-order symmetric tensor, defined by Qi, Wang and Wu
\cite{s10}, eigenvalues of general square tensors extended by Qi in
\cite{s13} Chang, Pearson and Zhang in \cite{s09} and equivalent
eigenvalue pair classes by Cartwright and Sturmfels \cite{CS10}.
Here, we are concerned with the one in \cite{s09, s13} as reviewed
below.
\begin{Definition}
Let $\textrm{C}$ be the complex field. A pair $(\lambda,x)\in
\textrm{C}\times (\textrm{C}^n\backslash\{0\})$ is called an
eigenvalue-eigenvector pair of $\mathcal{A}$, if they satisfy:
\begin{equation}\label{heigen}
 \mathcal{A}x^{m-1} = \lambda x^{[m-1]},
\end{equation}
where $n$-dimensional column vectors
$\mathcal{A}x^{m-1}$ and $x^{[m-1]}$ are defined as
\begin{eqnarray*}
\mathcal{A}x^{m-1} :=\left(\sum^n_{i_2,\ldots,i_m=1}A_{ii_2\cdots i_m}x_{i_2}\cdots x_{i_m}\right)_{1\le i\le n}\quad \text{and} \quad
\quad
x^{[m-1]}:=\left( x^{m-1}_i\right)_{1\le i\le n},
\end{eqnarray*}
respectively.
\end{Definition}

Nonnegative tensors, arising from multilinear pagerank \cite{s2},
spectral hypergraph theory \cite{BP, BP1, hq}, and higher-order
Markov chains \cite{s12}, etc., form a singularly important class of
tensors and have attracted more and more attention since they share
some intrinsic properties with those of the nonnegative matrices.
One of those properties is the Perron-Frobenius theorem on
eigenvalues. In \cite{s1}, Chang, Pearson, and Zhang generalized the
Perron-Frobenius theorem for nonnegative matrices to irreducible
nonnegative tensors. In \cite{fgh}, Friedland, Gaubert and Han
generalized the Perron-Frobenius theorem to weakly irreducible
nonnegative tensors. Further generalization of the Perron-Frobenius
theorem to nonnegative tensors can be found in \cite{pear, yy, yy1}.
Numerical methods for finding the spectral radius of nonnegative
tensors are subsequently proposed. Ng, Qi, and Zhou \cite{s12}
provided an iterative method to find the largest eigenvalue of an
irreducible nonnegative tensor by extending the Collatz method
\cite{colla} for calculating the spectral radius of an irreducible
nonnegative matrix. The Ng-Qi-Zhou method is efficient but it is not
always convergent for irreducible nonnegative tensors. Chang,
Pearson and Zhang \cite{cpz3} extended the notion of primitive
matrices into the realm of tensors, and established the convergence
of the Ng-Qi-Zhou method for primitive tensors. Zhang and Qi
\cite{zq} established global linear convergence of the Ng-Qi-Zhou
method for essentially positive tensors. Liu, Zhou and Ibrahim
\cite{lzi} proposed an always convergent algorithm for computing the
largest eigenvalue of an irreducible nonnegative tensors.    Zhang,
Qi, and Xu \cite{zqx} established its explicit linear convergence
rate for weakly positive tensors.

The essentially nonnegative tensor we defined in this paper is
ultimately related to the nonnegative tensor and includes the latter
one as a special case. It is a higher order generalization of the
so-called essentially nonnegative matrix, whose off-diagonal entries
are all nonnegative. Such a class of matrices possesses nice
properties on eigenvalues. It follows from the famous
Perron-Frobenius theorem for nonnegative matrices that for any
essentially nonnegative matrix $A$, there exists a real eigenvalue
with a nonnegative eigenvector, which is the largest one among real
parts of all other eigenvalues of $A$. This special eigenvalue,
termed as $r(A)$, is often called the \emph{dominant eigenvalue} of
$A$. Moreover, $r(A)$ is known as a convex function of the diagonal
entries of $A$. This convexity is a fundamental property for
essentially nonnegative matrices \cite{convex1, convex2, horn,
convex3} and has numerous applications, not only in many branches of
mathematics, such as graph theory \cite{graph}, differential
equations \cite{convex3}, but also in practical fields, e.g.,
population biology \cite{convex3}, demography \cite{demo}, and
analytic hierarchy process as well \cite{ahp}. A natural question
arises: does this convexity maintain for higher-order essentially
nonnegative tensors? In this paper, we will give an affirmative
answer to this question.

Similar to the essentially nonnegative matrix, an essentially
nonnegative tensor has a real eigenvalue with the property that it
is greater than or equal to the real part of every eigenvalue of
$\mathcal{A}$. We also call it the \emph{dominant eigenvalue} of
$\mathcal{A}$, and denoted by $\lambda(\mathcal{A})$. Particularly,
if $\mathcal{A}$ is nonnegative, we have
$\rho(\mathcal{A})=\lambda(\mathcal{A})$, where $\rho(\mathcal{A})$
is the spectral radius of $\mathcal{A}$. By employing the technique
proposed in \cite{convex3}, we manage to obtain that the dominant
eigenvalue is a convex function of the diagonal elements for any
essentially nonnegative tensor. In addition, it is also a convex
function of all elements of a tensor in some special convex set of
tensors. Furthermore, the log convexity is also exploited for
essentially nonnegative tensors with whose entries are either
identically zero or log convex of some real univariate functions.
Finally, we propose an algorithm to calculate the dominant
eigenvalue, convergence of the proposed algorithm is established and
numerical results are reported to show the effectiveness of the
proposed algorithm.

This paper is organized as follows. In Section 2, we recall some
preliminary results, introduce the concept of essentially
nonnegative tensors, and characterize some basic properties of such
tensors. In Section 3, we show that the spectral radius of
nonnegative tensors is a convex function of the diagonal elements,
and so is the dominant eigenvalue of essentially nonnegative
tensors.  Section 4 is devoted to the log convexity of the dominant
eigenvalue. In Section 5, we give an algorithm to calculate the
dominant eigenvalue, and some numerical results are reported. Some
concluding remarks are made in Section 6.

\Section{Preliminaries and essentially nonnegative tensors}

We start this section with some fundamental notions and properties
on tensors. An $m$-order $n$-dimensional tensor $\mathcal{A}$ is
called nonnegative (or, respectively, positive) if $A_{i_1\cdots
i_m}\ge 0$ (or, respectively, $A_{i_1\cdots i_m}> 0$). The $m$-order
$n$-dimensional unit tensor, denoted by $\mathcal{I}$, is the tensor
whose entries are $\delta_{i_1\ldots i_m}$ with $\delta_{i_1\ldots
i_m}=1$ if and only if $i_1=\cdots =i_m$ and otherwise zero. The
symbol $\mathcal{A}\ge \mathcal{B}$ means that
$\mathcal{A}-\mathcal{B}$ is a nonnegative tensor. A tensor
$\mathcal{A}$ is called reducible, if there exists a nonempty proper
index subset $I\subset \{1,2,\ldots,n\}$ such that
\begin{displaymath}
A_{i_1\cdots i_m} =0,\quad \forall i_1\in I,\quad \forall
i_2,\ldots,i_m\not\in I.
\end{displaymath}
Otherwise, we say $\mathcal{A}$ is irreducible. We call
$\rho(\mathcal{A})$ the spectral radius of tensor $\mathcal{A}$ if
\begin{displaymath}
\rho(\mathcal{A})=\max\{|\lambda|:\, \text{$\lambda$ is an
eigenvalue of $\mathcal{A}$}\},
\end{displaymath}
where $|\lambda|$ denotes the modulus of $\lambda$. An immediate
consequence on the spectral radius follows directly from Corollary 3
in \cite{s4}.

\begin{Lemma}\label{qi}
Let $\mathcal{A}$ be an $m$-order $n$-dimensional tensor. Suppose
that $\mathcal{B}=a(\mathcal{A}+b\mathcal{I})$, where $a$ and $b$
are two real numbers. Then $\mu$ is an eigenvalue of $\mathcal{B}$
if and only if $\mu=a(\lambda+b)$ and $\lambda$ is an eigenvalue of
$\mathcal{A}$. In this case, they have the same eigenvectors.
Moreover, $\rho(\mathcal{B})\le
|a|\left(\rho(\mathcal{A})+|b|\right)$.
\end{Lemma}

Let $P:=\{x\in \textrm{R}^n:\, x_i\ge 0,\, 1\le i\le n\}$, and
$\mathrm{int}(P)=\{x\in \textrm{R}^n:\, x_i> 0,\, 1\le i\le n\}$.
The Perron-Frobenius theorem for nonnegative tensors is as below,
following by \cite[Theorem 1.4]{s1}.
\begin{Theorem}   \label{thmpf}
If $\mathcal{A}$ is an irreducible nonnegative tensor of order $m$
and dimension $n$, then there exist $\lambda_0 > 0$ and $x_0\in
\mathrm{int}(P)$ such that
\begin{equation*}
\mathcal{A}x_0^{m-1} = \lambda_0 x^{[m-1]}_0.
\end{equation*}
Moreover, if $\lambda$ is an eigenvalue with a nonnegative
eigenvector, then $\lambda=\lambda_0$. If $\lambda$ is an eigenvalue
of $\mathcal{A}$, then $|\lambda|\le \lambda_0$.
\end{Theorem}

The well-known Collatz minimax theorem \cite{colla} for irreducible
nonnegative matrices has been extended to irreducible nonnegative
tensors in \cite[Theorem 4.2]{s1}. 
\begin{Theorem}  \label{thmmm}
Assume that  $\mathcal{A}$ is an irreducible nonnegative tensor of
order m dimension n. Then
\begin{equation*}
\min_{x\in
{\mathrm{int}}(P)}\max_{x_i>0}\dfrac{(\mathcal{A}x^{m-1})_i}{x^{m-1}_i}=\lambda_0=
\max_{x\in
{\mathrm{int}}(P)}\min_{x_i>0}\dfrac{(\mathcal{A}x^{m-1})_i}{x^{m-1}_i},
\end{equation*}
where $\lambda_0$ is the unique positive eigenvalue corresponding to
a positive eigenvector.
\end{Theorem}

For nonnegative tensors, Yang and Yang \cite{yy1} asserted that the
spectral radius is an eigenvalue, which is a generalization of the
weak Perron-Frobenius theorem for nonnegative matrices. We state it
\cite[Theorem 2.3 and Lemma 5.8]{yy1} in the following theorem.
\begin{Theorem}   \label{spectral}
Assume that  $\mathcal{A}$ is a nonnegative tensor of order m
dimension n, then $\rho(\mathcal{A})$ is an eigenvalue of
$\mathcal{A}$ with a nonzero nonnegative eigenvector. Moreover, for
any $x\in \mathrm{int}(P)$ we have
\begin{equation*}
\min_{1\le i\le n}\dfrac{(\mathcal{A}x^{m-1})_i}{x^{m-1}_i}\le \rho(\mathcal{A})\le
\max_{1\le i\le n}\dfrac{(\mathcal{A}x^{m-1})_i}{x^{m-1}_i}.
\end{equation*}
\end{Theorem}
The following inequality and continuity of the spectral radius were
given in \cite[Lemma 3.5]{yy1} and the proof of \cite[Theorem
2.3]{yy1}, respectively.
\begin{Lemma}   \label{spectral1}
Let  $\mathcal{A}$ be a nonnegative tensor of order m and dimension
n, and $\varepsilon>0$ be a sufficiently small number. Suppose
$\mathcal{A}\le \mathcal{B}$, then
$\rho(\mathcal{A})\le\rho(\mathcal{B})$. Furthermore, if
$\mathcal{A}_\varepsilon= \mathcal{A}+ \mathcal{E}$ where
$\mathcal{E}$ denotes the tensor with every entry being
$\varepsilon$, then
\begin{equation*}
\lim_{\varepsilon\to 0}\rho(\mathcal{A}_\varepsilon)=\rho(\mathcal{A}).
\end{equation*}
\end{Lemma}
Based on the above results, we can easily get the following lemma.
\begin{Lemma} \label{lem1.1} Suppose that $\mathcal{A}$ is an irreducible nonnegative tensor of
order $m$ dimension $n$ and that there exists a nonzero vector $x\in P$ and a real number $\beta$ such that
\begin{equation}\label{equ1.1}
\mathcal{A}x^{m-1}\le \beta x^{[m-1]}.
\end{equation}
Then $\beta>0$, $x\in \mathrm{int}(P)$, and
$\rho(\mathcal{A})\le\beta$. Furthermore, $\rho(\mathcal{A})<\beta$
unless equality holds in (\ref{equ1.1}).
\end{Lemma}
\begin{proof} Assume on the contrary that for $x\in \mathrm{int}(P)$ there exists a nonempty proper
index subset $I\subset\{1,2,\ldots,n\}$ such that $x_i=0$ for $i\in
I$ and $x_i>0$ for $i\not\in I$. It follows from (\ref{equ1.1}) that
\begin{displaymath} A_{i_1\cdots i_m} =0,\quad \forall i_1\in
I,\quad \forall i_2,\ldots,i_m\not\in I.
\end{displaymath}
A contradiction to the irreducibility of $\mathcal{A}$ comes, which
henceforth implies that $x\in \mathrm{int}(P)$. Together with Lemma
2.2 in \cite{s12}, $\mathcal{A}x^{m-1}\in \mathrm{int}(P)$ is
established. It further deduces that $\beta>0$, and then the last
statement holds from Lemma 5.9 in \cite{yy1}. This completes the
proof.
\end{proof}

A simple but useful result follows immediately from Lemmas
\ref{spectral1} and \ref{lem1.1}.
\begin{Lemma} \label{lem1.2} Let $\mathcal{A}$ and $\mathcal{B}$ be irreducible nonnegative tensors of
order m dimension n. If $\mathcal{A}\le \mathcal{B}$ and
$\mathcal{A}\ne\mathcal{B}$, then
 $\rho(\mathcal{A})< \rho(\mathcal{B})$.
\end{Lemma}
\begin{proof}
By Lemma \ref{spectral1}, $\rho(\mathcal{A})\le \rho(\mathcal{B})$.
Since $\mathcal{B}$ is irreducible, Theorem \ref{thmpf} implies that
there exists $x\in \mathrm{int}(P)$ such that
\begin{equation}\label{eq1.2}
\mathcal{A}x^{m-1}\le \mathcal{B}x^{m-1}=\rho(\mathcal{B})x^{[m-1]}.
\end{equation}
Since $x\in \mathrm{int}(P)$ and  $\mathcal{A}\ne\mathcal{B}$,
equality cannot hold in (\ref{eq1.2}). The desired strict inequality
$\rho(\mathcal{A})< \rho(\mathcal{B})$ holds from Lemma
\ref{lem1.1}.
\end{proof}

The remaining of this section is devoted to the essentially
nonnegative tensor, with the introduction of its definition and some
basic properties.
\begin{Definition}
Let $\mathcal{A}$ be an $m$-order and $n$-dimensional tensor.
$\mathcal{A}$ is said to be \emph{essentially nonnegative} if all
its off-diagonal entries are nonnegative.
\end{Definition}

\begin{Theorem}\label{thm1.1}
Let $\mathcal{A}$ be an $m$-order and $n$-dimensional essentially
nonnegative tensor. Then there exists $\alpha>0$ such that
$\alpha\mathcal{I}+\mathcal{A}$ is nonnegative. Moreover,
$\mathcal{A}$ has a real eigenvalue $\lambda(\mathcal{A})$ with
corresponding eigenvector in $P$ and $\lambda(\mathcal{A})\ge
\mathrm{Re}\lambda$ for every eigenvalue $\lambda$ of $\mathcal{A}$.
Furthermore,
$$ \lambda(\mathcal{A})=\rho(\alpha\mathcal{I}+\mathcal{A})-\alpha.$$
\end{Theorem}
\begin{proof}
Take
$$\alpha=\max_{1\le i\le n}|A_{i\ldots i}|+1.$$
Clearly, $\alpha>0$ and $\alpha\mathcal{I}+\mathcal{A}$ is
nonnegative. By Lemma \ref{qi} and Theorem \ref{spectral}, we have
\begin{equation}\label{equ1.3}
\rho(\alpha\mathcal{I}+\mathcal{A})=\alpha+\lambda_1,
\end{equation}
where $\lambda_1$ is an eigenvalue of $\mathcal{A}$ with
corresponding eigenvector in $P$. Thus, (\ref{equ1.3}) implies
$\lambda_1\in \mathrm{R}$. Let $\lambda(\mathcal{A})=\lambda_1$, It
follows from Lemma \ref{qi} that,
\begin{eqnarray*}
 \lambda(\mathcal{A})+\alpha &=& \max\{|\alpha+\lambda|: \, \text{$\lambda$ is an eigenvalue of $\mathcal{A}$}\}\\
 &\ge &|\alpha+\lambda|\ge \alpha+\mathrm{Re}\lambda.
 \end{eqnarray*}
The desired result arrives.
\end{proof}

We call such an eigenvalue in the above theorem the \emph{dominant
eigenvalue} of $\mathcal{A}$. Throughout this paper,
$\rho(\mathcal{A})$ and $ \lambda(\mathcal{A})$ will denote the
spectral radius and dominant eigenvalue respectively of a tensor $
\mathcal{A}$. In the next section, we will show that both
$\rho(\mathcal{A})$ and $ \lambda(\mathcal{A})$ are convex functions
of the diagonal elements of $\mathcal{A}$.

\Section{Convexity of the spectral radius and the dominant
eigenvalue}

Based on Theorems \ref{thmpf} and \ref{spectral}, we proceed with
the convexity of the dominant eigenvalue of essentially nonnegative
tensors in this section. It can be verified that the diagonal
entries have nothing to do with the irreducibility of a tensor.
Specifically, let $\mathcal{A}$ be an essentially nonnegative tensor
of order $m$ and dimension $n$, define a nonnegative tensor
$\mathcal{B}$ by $B_{i_1\ldots i_m}=0$ if $i_1=\cdots =i_m$ and the
others are $A_{i_1\ldots i_m}$. Then $\mathcal{A}$ is irreducible if
and only if $\mathcal{B}$ is. Equivalently, $\mathcal{A}$ is
irreducible if and only if $\mathcal{A}+\alpha\mathcal{I}$ is,
whenever it is nonnegative. Thus, by Lemma \ref{spectral1} and
Theorem \ref{thm1.1}, it is sufficient to consider the class of
irreducible nonnegative tensors.

\begin{Theorem}\label{thmcov}
If $\mathcal{A}$ is a given irreducible nonnegative tensor of order
$m$ and dimension $n$, and $\mathcal{D}$ is allowed to vary in the
class of nonnegative diagonal tensors, then the spectral radius
$\rho(\mathcal{A}+\mathcal{D})$ is a convex function of the diagonal
entries of $\mathcal{D}$. That is, for nonnegative diagonal tensors
$\mathcal{C}$ and $\mathcal{D}$ we have
\begin{equation}\label{eq3.1}
\rho(\mathcal{A}+t\mathcal{C}+(1-t)\mathcal{D})\le t\rho(\mathcal{A}+\mathcal{C})+(1-t)\rho(\mathcal{A}+\mathcal{D}),\quad \forall t\in [0,1].
\end{equation}
Moreover, equality holds in (\ref{eq3.1}) for some $t\in (0,1)$ if
and only if $\mathcal{D}-\mathcal{C}$ is a scalar multiple of the
unite tensor $\mathcal{I}$.
\end{Theorem}
\begin{proof}
Since both $\mathcal{A}+\mathcal{C}$ and $\mathcal{A}+\mathcal{D}$
are irreducible nonnegative tensors, by Theorem \ref{thmpf} and
Theorem \ref{spectral} we have $\rho(\mathcal{A}+\mathcal{C})>0$,
$\rho(\mathcal{A}+\mathcal{D})>0$, and there exist $x,y\in
\mathrm{int}(P)$ such that
\begin{equation*}
(\mathcal{A}+\mathcal{C})x^{m-1}=\rho(\mathcal{A}+\mathcal{C})x^{[m-1]},\quad
(\mathcal{A}+\mathcal{D})y^{m-1}=\rho(\mathcal{A}+\mathcal{D})y^{[m-1]}.
\end{equation*}
That is, for $i=1,2,\ldots,n$ we have
\begin{eqnarray*}
\rho(\mathcal{A}+\mathcal{C})&=& C_{i\ldots i}+\sum_{i_2\ldots i_m=1}^n A_{i\,i_2\ldots i_m}
\dfrac{x_{i_2}\cdots x_{i_m}}{x_i},\\
\rho(\mathcal{A}+\mathcal{D})&=& D_{i\ldots i}+\sum_{i_2\ldots i_m=1}^n A_{i\,i_2\ldots i_m}
\dfrac{y_{i_2}\cdots y_{i_m}}{y_i},
\end{eqnarray*}
and hence $\rho(\mathcal{A}+\mathcal{C})-C_{i\ldots i}>0$ and
$\rho(\mathcal{A}+\mathcal{D})-D_{i\ldots i}>0$. The inequality
between geometric and arithmetic means yields
\begin{eqnarray}
\left(\sum_{i_2\ldots i_m=1}^n A_{i\,i_2\ldots i_m}\dfrac{x_{i_2}\cdots x_{i_m}}{x_i}\right)^t
\left(\sum_{i_2\ldots i_m=1}^n A_{i\,i_2\ldots i_m}\dfrac{y_{i_2}\cdots y_{i_m}}{y_i}\right)^{1-t}
\le t(\rho(\mathcal{A}+\mathcal{C})-C_{i\ldots i})\nonumber\\
+(1-t)(\rho(\mathcal{A}+\mathcal{D})- D_{i\ldots i}).\label{eq3.2}
\end{eqnarray}
Therefore, H\"{o}lder's inequality and Theorem \ref{thmmm} give from (\ref{eq3.2})
\begin{eqnarray*}
\rho(\mathcal{A}+t\mathcal{C}+(1-t)\mathcal{D})&\le & \max_{1\le i\le n}
\left\{tC_{i\ldots i}+(1-t)D_{i\ldots i}+\sum_{i_2\ldots i_m=1}^n A_{i\,i_2\ldots i_m}
\dfrac{z_{i_2}\cdots z_{i_m}}{z_i}\right\}\\ &\le & t\rho(\mathcal{A}+\mathcal{C})+(1-t)\rho(\mathcal{A}+\mathcal{D}),
\end{eqnarray*}
where $z_i=x_i^ty_i^{1-t}$ for $i=1,\ldots,n$. This shows (\ref{eq3.1}) holds.

The inequality between geometric and arithmetic means implies that
equality in (\ref{eq3.1}) holds for $t\in (0,1)$ if and only if
$\rho(\mathcal{A}+\mathcal{C})- C_{i\ldots
i}=\rho(\mathcal{A}+\mathcal{D})- D_{i\ldots i}$ for $i=1,\ldots,n$,
i.e., $\mathcal{D}-\mathcal{C}=\gamma\mathcal{I}$ where
$\gamma=\rho(\mathcal{A}+\mathcal{D})-\rho(\mathcal{A}+\mathcal{C})$.
This completes the proof.
\end{proof}

The convexity involved in Theorem \ref{thmcov} can be extended to
the case of essentially nonnegative tensors as follows.
\begin{Corollary}\label{corcov}
If $\mathcal{A}$ is a given irreducible essentially nonnegative
tensor of order $m$ dimension $n$ and $\mathcal{D}$ is allowed to
vary in the class of diagonal tensors, then the dominant eigenvalue
$\lambda(\mathcal{A}+\mathcal{D})$ is a convex function of the
diagonal entries of $\mathcal{D}$. That is, for diagonal tensors
$\mathcal{C}$ and $\mathcal{D}$ we have
\begin{equation}\label{eq3.3}
\lambda(\mathcal{A}+t\mathcal{C}+(1-t)\mathcal{D})\le t\lambda(\mathcal{A}+
\mathcal{C})+(1-t)\lambda(\mathcal{A}+\mathcal{D}),\quad \forall t\in [0,1].
\end{equation}
Moreover, equality holds in (\ref{eq3.3}) for some $t\in (0,1)$ if
and only if $\mathcal{D}-\mathcal{C}$ is a scalar multiple of the
unite tensor $\mathcal{I}$.
\end{Corollary}
\begin{proof} Take $$\alpha=1+\max_{1\le i\le n}\{|A_{i\ldots i}|+|C_{i\ldots i}|+|D_{i\ldots i}|\}.$$
Then $\alpha\mathcal{I}+\mathcal{A}+\mathcal{C}$ and
$\alpha\mathcal{I}+\mathcal{A}+\mathcal{D}$ are all irreducible
nonnegative tensors. By Theorem \ref{thm1.1} and Theorem
\ref{thmcov}, we have for $0\le t\le 1$
\begin{eqnarray*}
\lambda(\mathcal{A}+t\mathcal{C}+(1-t)\mathcal{D})+\alpha&= & \rho(\alpha\mathcal{I}+
\mathcal{A}+t\mathcal{C}+(1-t)\mathcal{D}) \\ &\le &t\rho(\alpha\mathcal{I}+\mathcal{A}
+\mathcal{C})+(1-t)\rho(\alpha\mathcal{I}+\mathcal{A}+\mathcal{D})\\
&=& t\lambda(\mathcal{A}+\mathcal{C})+(1-t)\lambda(\mathcal{A}+\mathcal{D})+\alpha,
\end{eqnarray*}
which yields (\ref{eq3.3}). This completes the proof.
\end{proof}

Invoking the continuity presented in Lemma \ref{spectral1}, it is
easy to see that Theorem \ref{thmcov} and Corollary \ref{corcov}
hold even when $\mathcal{A}$ is reducible. Moreover, Theorem
\ref{thmcov} and Corollary \ref{corcov} give necessary and
sufficient conditions for the strict convexity. It is worth pointing
out that the convexity of the dominant eigenvalue only works on the
diagonal elements other than on all elements of the essentially
nonnegative tensor, unless for some special cases. By collecting all
symmetric essentially nonnegative tensors of order $m$ and dimension
$n$, we can get a closed convex cone, says ${\cal{S}}(m,n)$. The
dominant eigenvalue of any tensor in ${\cal{S}}(m,n)$ remains convex
of all elements of the corresponding tensor in the domain
${\cal{S}}(m,n)$, as the following proposition shows.

\begin{Proposition} For any $\mathcal{A}$, ${\mathcal{B}}\in
{\cal{S}}(m,n)$, and any $t\in[0,1]$, we have
$$\lambda(t{\mathcal{A}}+(1-t){\mathcal{B}})\leq t \lambda({\mathcal{A}})+(1-t)\lambda({\mathcal{B}}).$$
\end{Proposition}
\begin{proof} For any $\mathcal{A}$, ${\mathcal{B}}\in
{\cal{S}}(m,n)$, there exists an integer $k>0$ such that
${\mathcal{A}}+k{\mathcal{I}}$ and ${\mathcal{B}}+k{\mathcal{I}}$
are nonnegative and symmetric and hence for any of their convex
combinations. The Perron-Frobenius theorem then ensures that
$\rho({\mathcal{A}}+k{\mathcal{I}})$,
$\rho({\mathcal{B}}+k{\mathcal{I}})$ and
$\rho(t{\mathcal{A}}+(1-t){\mathcal{B}}+k{\mathcal{I}})$
($t\in[0,1]$) all act as eigenvalues of the corresponding
nonnegative symmetric tensor. By the variational approach, it
follows that
\begin{eqnarray*}
  & &  \rho(t{\mathcal{A}}+(1-t){\mathcal{B}}+k{\mathcal{I}}) \nonumber\\
    &=&  \max\left\{(t{\mathcal{A}}+(1-t){\mathcal{B}}+k{\mathcal{I}})x^m:
\sum_{i=1}^n x_i^m=1\right\}  \nonumber\\
    &\leq & t\max\left\{({\mathcal{A}}+k{\mathcal{I}})x^m:
\sum_{i=1}^n x_i^m=1\right\}+(1-t)\max\left\{({\mathcal{B}}+k{\mathcal{I}})x^m:
\sum_{i=1}^n x_i^m=1\right\} \nonumber\\
&=&
t\rho({\mathcal{A}}+k{\mathcal{I}})+(1-t)\rho({\mathcal{B}}+k{\mathcal{I}}).\nonumber
\end{eqnarray*}
Combining with the fact that
$\rho({\mathcal{A}}+k{\mathcal{I}})=\lambda({\mathcal{A}})+k$, the
desired convexity follows.
\end{proof}

\Section{Log convexity of the spectral radius and the dominant eigenvalue}

If a function $f(x)$ is positive on its domain and $\log f(x)$ is
convex, then $f(x)$ is called \emph{log convex}. It is known
\cite{king} that the sum or product of log convex functions is also
log convex. In this section we extend Kingman's theorem \cite{king}
for matrices to tensors. Our motivation for the following proof
comes from \cite{convex3}.

\begin{Theorem} \label{logcov}
For $t\in [0,1]$ assume that $\mathcal{F}(t)=\left(F_{i_1\ldots
i_m}(t)\right)$ is an $m$-order $n$-dimensional irreducible
nonnegative tensor, and suppose that for $1\le i_1,\ldots, i_m\le
n$, $F_{i_1\ldots i_m}(t)$ is either identically zero or positive
and a log convex function of $t$. Then $\rho(\mathcal{F}(t))$ is a
log convex function of $t$ for $t\in [0,1]$. That is, if
$\mathcal{F}(0)=\mathcal{A}$, $\mathcal{F}(1)=\mathcal{B}$, and a
nonnegative tensor $\mathcal{G}(t)=\left(A_{i_1\ldots
i_m}^{1-t}B_{i_1\ldots i_m}^t\right)$, then
 \begin{equation}\label{eq4.1}
\rho(\mathcal{F}(t))\le \rho(\mathcal{G}(t))\le
\rho(\mathcal{A})^{1-t}\rho(\mathcal{B})^t.
\end{equation}
Moreover, the first equality occurs in (\ref{eq4.1}) for some $t$
with $t\in(0,1)$ if and only if $$\mathcal{F}(t)=\mathcal{G}(t),$$
and the second equality occurs in (\ref{eq4.1}) for some $t$ with
$t\in (0,1)$ if and only if there exists a constant $\sigma>0$ and a
positive diagonal matrix $D=diag(d_1,\ldots,d_n)$ such that
 \begin{equation*}
 \mathcal{B}=\sigma\mathcal{A}\cdot D^{-(m-1)}\cdot \overbrace{D\cdots D}^{m-1}\,\, \text{with $B_{i_1i_2\ldots i_m}=\sigma A_{i_1i_2\ldots i_m} d_{i_1}^{-(m-1)}d_{i_2}\cdots d_{i_m}$}.
\end{equation*}
\end{Theorem}
\begin{proof} Clearly, $\mathcal{G}(0)=\mathcal{F}(0)=\mathcal{A}$ and $\mathcal{G}(1)=\mathcal{F}(1)=\mathcal{B}$.
The log convexity assumption on $F_{i_1\ldots i_m}(t)$ implies that,
for $t\in[0,1]$,
 \begin{equation*}
 \mathcal{F}(t)\le \mathcal{G}(t),
 \end{equation*}
 which, together with Lemma \ref{spectral1}, implies
 \begin{equation}\label{eq4.2}
 \rho(\mathcal{F}(t))\le \rho(\mathcal{G}(t)).
 \end{equation}
 Since $\mathcal{F}(t)$ is irreducible, if equality holds in \ref{eq4.2} for some $t_0$ with $0<t_0<1$, Lemma \ref{lem1.2} implies that $\mathcal{F}(t_0)=\mathcal{G}(t_0)$.

 Since $\mathcal{F}(0)$ and  $\mathcal{F}(1)$ are irreducible nonnegative, Theorem \ref{thmpf} shows that there exist $x,y\in \mathrm{int}(P)$ such that
 $$
 \mathcal{A}x^{m-1}=\rho(\mathcal{A})x^{[m-1]},\quad \quad \mathcal{B}y^{m-1}=\rho(\mathcal{B})y^{[m-1]}.
 $$
For a fixed $t\in(0,1)$, define $z=x^{1-t}y^t$, i.e.,
$z_i=x_i^{1-t}y_i^t$ for $1\le i\le n$. Then, the $i$th component of
$\mathcal{G}(t)z^{m-1}$ satisfies
\begin{equation*}
\left(\mathcal{G}(t)z^{m-1}\right)_i=\sum_{i_2\ldots i_m=1}^n
A_{i\,i_2\ldots i_m}^{1-t}B_{i\,i_2\ldots i_m}^{t}z_{i_2}\cdots
z_{i_m}.
\end{equation*}
Hence, H\"{o}lder's inequality gives
\begin{eqnarray}
\left(\mathcal{G}(t)z^{m-1}\right)_i&\le& \left(\sum_{i_2\ldots i_m=1}^n A_{i\,i_2\ldots i_m}x_{i_2}\cdots x_{i_m}\right)^{1-t}\left(\sum_{i_2\ldots i_m=1}^n B_{i\,i_2\ldots i_m}y_{i_2}\cdots y_{i_m}\right)^{t}\nonumber\\
&=& \rho(\mathcal{A})^{1-t}\rho(\mathcal{B})^tz_i^{m-1}.
\label{eq4.3}
\end{eqnarray}
It follows from Lemma \ref{lem1.1} and (\ref{eq4.3}) that
$$\rho(\mathcal{G}(t))\le \rho(\mathcal{A})^{1-t}\rho(\mathcal{B})^t.$$
Furthermore, equality holds in \ref{eq4.3} for some $t\in (0,1)$ if
and only if, for $1\le i\le n$,
\begin{equation}\label{eq4.4}
B_{i\,i_2\ldots i_m}y_{i_2}\cdots y_{i_m}=\sigma_iA_{i\,i_2\ldots
i_m}x_{i_2}\cdots x_{i_m}.
\end{equation}
Summing (\ref{eq4.4}) over $i_2\ldots i_m$ yields
\begin{equation}\label{eq4.5}
\rho(\mathcal{B})y^{m-1}_i=\sigma_i\rho(\mathcal{A})x^{m-1}_i.
\end{equation}
Take
\begin{equation*}
\sigma=\dfrac{\rho(\mathcal{B})}{\rho(\mathcal{A})},\quad
d_i=\dfrac{x_i}{y_i},
\end{equation*}
Then, combining (\ref{eq4.4}) and (\ref{eq4.5}) we obtain
\begin{equation*}
B_{i\,i_2\ldots i_m}=\sigma A_{i\,i_2\ldots
i_m}d_i^{-(m-1)}d_{i_2}\cdots d_{i_m},
\end{equation*}
i.e.,
\begin{equation*}
 \mathcal{B}=\sigma\mathcal{A}\cdot D^{-(m-1)}\cdot \overbrace{D\cdots D}^{m-1}.
 \end{equation*}
This completes the proof.
\end{proof}

By Theorems \ref{spectral} and \ref{thm1.1}, the above theorem also
holds for the dominant eigenvalue of $\mathcal{F}(t)$, when
$\mathcal{F}(t)$ is essentially nonnegative with $t\in[0,1]$.

\Section{An algorithm for calculating the dominant eigenvalue}

Let $\mathcal{A}$ be an essentially nonnegative tensor of order $m$
and dimension $n$. In this section we propose an algorithm to
calculate the dominant eigenvalue of an essentially nonnegative
tensor. This algorithm is a modification of the Ng-Qi-Zhou
algorithm given in \cite{s12}. By Lemma \ref{spectral1} and Theorem
\ref{thm1.1}, we modify the Ng-Qi-Zhou algorithm such that for any
essentially nonnegative tensor, the sequence generated by the
modified algorithm always converges to its dominant eigenvalue.

Define two functions from $\mathrm{int}(P)$ to
$P$:
\begin{equation}\label{funct}
F(x):=\min_{x_i\ne 0}\dfrac{(\mathcal{W}x^{m-1})_i}{x^{m-1}_i},\quad
G(x):= \max_{x_i\ne 0}\dfrac{(\mathcal{W}x^{m-1})_i}{x^{m-1}_i},
\end{equation}
where $\mathcal{W}$ is an irreducible nonnegative tensor. The details of the modified algorithm are given as follows.

\textbf{Algorithm 5.1:}

\alglist
\item[{\bf Step 0.}]  Given a sufficiently small number $\varepsilon>0$. Let
\begin{equation}\label{wtensor}
\mathcal{W}=\mathcal{A}+\alpha\mathcal{I}+\mathcal{E},
\end{equation}
 where $$\alpha=\max_{1\le i\le n}|A_{i\ldots i}|+1,$$
and $\mathcal{E}$ is the tensor with every entry being
$\varepsilon$. Choose any $x^{(0)}\in \mathrm{int}(P)$. Set $y^{(0)}=\mathcal{W}\left(x^{(0)}\right)^{m-1}$ and $k:=0$.

\item[{\bf Step 1.}] Compute
\begin{equation*}
x^{(k+1)}=\dfrac{\left(y^{(k)}\right)^{[\frac1{m-1}]}}{\left\|\left(y^{(k)}\right)^{[\frac1{m-1}]}\right\|},\quad \quad y^{(k+1)}=\mathcal{B}\left(x^{(k+1)}\right)^{m-1}.
\end{equation*}
According to (\ref{funct}), compute $F(x^{(k+1)})$ and $G(x^{(k+1)})$.

\item[{\bf Step 2.}]
If $G(x^{(k+1)})-F(x^{(k+1)})< \varepsilon$, stop. Output $\varepsilon$-approximation of the
dominant eigenvalue of $\mathcal{A}$:
\begin{equation}\label{eigenapp}\lambda^{(k+1)}=\dfrac12\left(G(x^{(k+1)})+F(x^{(k+1)})\right)-\alpha,\end{equation}
and the corresponding eigenvector $x^{(k+1)}$. Otherwise, set
$k:=k+1$ and go to Step 1.
\end{list}

Clearly, the tensor $\mathcal{W}$ defined by (\ref{wtensor}) is
positive and hence it is primitive. By Theorems \ref{thmpf} and
\ref{thmmm}, Algorithm 5.1 is well-defined. As an immediate
consequence of Lemma \ref{spectral1}, Theorem \ref{thm1.1}, and
Theorem 3.3 in \cite{pear}, we have the following convergence
theorem.
\begin{Theorem}\label{convergefg}
Let $\mathcal{A}$ be an essentially nonnegative tensor of order $m$ and
dimensional $n$, and let $\mathcal{W}$ be defined by (\ref{wtensor}) where $\varepsilon$
is a sufficiently small number. Then the sequences
$\{F(x^{(k)})\}$ and $\{G(x^{(k)})\}$, generated by Algorithm 5.1,
converge to $\lambda_\varepsilon$, where $\lambda_\varepsilon$ is the unique positive
eigenvalue of $\mathcal{W}$. Moreover, the sequence $\{x^{(k)}\}$
converges to $x_\varepsilon^*$ and $x_\varepsilon^*$ is a positive eigenvector of
$\mathcal{W}$ corresponding to the largest eigenvalue $\lambda_\varepsilon$. Furthermore,
$$
\lim_{\varepsilon\to 0}\lambda_\varepsilon=\lambda^*, \quad \lim_{\varepsilon\to 0}x^*_\varepsilon=x^*,
$$
where $\lambda^*$ is the spectral radius of $\mathcal{A}+\alpha\mathcal{I}$ and $x^*$ is the corresponding eigenvector. In particular, the dominant eigenvalue of $\mathcal{A}$ is $\lambda(\mathcal{A})=\lambda^*-\alpha$  and $x^*$ is also the
eigenvector corresponding to $\lambda(\mathcal{A})$.
\end{Theorem}
\begin{proof}
It follows from (\ref{wtensor}) that $\mathcal{W}$ is positive, and hence it is irreducible. Therefore, for any nonzero
$x\in P$, we have $\mathcal{W}x^{m-1}\in \mathrm{int}(P)$, which shows that the tensor $\mathcal{W}$ is primitive.
  Hence, by Theorem 3.3 in \cite{pear},
$$\lim_{k\to \infty}F(x^{(k)})=\lim_{k\to \infty}\{G(x^{(k)})= \lambda_\varepsilon,\quad
\lim_{k\to \infty}x^{(k)}=x^*_\varepsilon.
$$
Therefore, $\lambda_\varepsilon-\alpha$ is an $\varepsilon$-approximation of the dominant eigenvalue of $\mathcal{A}$ from Theorem \ref{thm1.1}. Furthermore, It follows from Lemma \ref{spectral1}  that
$$
\lim_{\varepsilon\to 0}\lambda_\varepsilon=\lambda^*,\quad \lim_{\varepsilon\to 0}x^*_\varepsilon=x^*.
$$
It is easy to see that $\lambda^*-\alpha$ is the dominant eigenvalue of $\mathcal{A}$ with corresponding eigenvector $x^*$.
\end{proof}

The above theorem shows that the convergence of Algorithm 5.1 is established for any essentially nonnegative tensor without the irreducible and primitive assumption. In order to show the effectiveness of Algorithm 5.1, we used \textsc{Matlab} 7.4 to test it on the following three examples.
\begin{Example} \label{exam1} Consider the $3$-order $3$-dimensional essentially nonnegative tensor
$$
\mathcal{A}=[A(1,:,:), A(2,:,:), A(3,:,:)],
$$
where
\begin{eqnarray*}
A(:,:,1) &=& \left(\begin{array}{ccc}
-1.51 & 8.35 & 1.03\\
4.04 & 3.72 & 1.45\\
6.71 & 6.43 & 1.35\\
\end{array}\right)\\
A(:,:,2) &=& \left(\begin{array}{ccc}
9.02 & 0.78 & 6.89\\
9.71 & -5.32 & 1.85\\
2.09 & 4.17 & 2.98\\
\end{array}\right)\\
A(:,:,3) &=& \left(\begin{array}{ccc}
9.55 & 1.57 & 6.91\\
5.63 & 5.55 & 1.43\\
5.76 & 8.29 & -0.15\\
\end{array}\right)
\end{eqnarray*}
\end{Example}
\begin{Example}\label{exam2}
 Let a $3$-order $3$-dimensional tensor $\mathcal{A}$ be defined by $A_{133}=A_{233}=A_{311}=A_{322}=1$, $A_{111}=A_{222}=-1$ and zero otherwise.
 \end{Example}
\begin{Example}\label{exam3}
 Let a $3$-order $4$-dimensional tensor $\mathcal{A}$ be defined by $A_{111}=A_{222}=A_{333}=A_{444}=-1$, $A_{112}=A_{114}=A_{121}=A_{131}=A_{212}=A_{332}=A_{443}=1$, and zero otherwise.
 \end{Example}
Clearly, the essentially nonnegative tensors defined as Examples \ref{exam1} and \ref{exam2} respectively are irreducible. While, the essentially nonnegative tensor defined as Example \ref{exam3} is reducible.
We take $\varepsilon=10^{-9}$ and terminate our iteration when one of the conditions $G(x^{(k)})-F(x^{(k)})\le 10^{-9}$ and $k\ge 100$ is satisfied. Algorithm 5.1 produces the dominant eigenvalue $\lambda(\mathcal{A})=36.2757$ with eigenvector $x^*=(1.0000;0.8351;0.9415)$ for Example \ref{exam1}, the dominant eigenvalue $\lambda(\mathcal{A})=1$ with eigenvector $x^*=(0.5000;0.5000;1.000)$ for Example \ref{exam2}, and the dominant eigenvalue $\lambda(\mathcal{A})=0.8225$ with eigenvector $x^*=(1.0000; 0.7408; 0.9714;0.5330)$ for Example \ref{exam3}. The details of numerical results are reported in Tables 1 and 2.
 We list the output details at each iteration for Example \ref{exam1} in Table 1.  We also report the number of iterations (\emph{No.Iter}), the elapsed CPU time (\emph{CPU(sec)}), the lower bound
$\underline{\lambda}^{(k)}=F(x^{(k)})-\alpha$ and the upper bound $\overline{\lambda}^{(k)}=G(x^{(k)})-\alpha$ for $k\ge 1$, the error $\Delta^{(k)}=\|\mathcal{A}(x^{(k)})^{m-1}-\lambda^{(k)}(x^{(k)})^{[m-1]}\|_{\infty}$, and the approximation $\lambda^{(k)}$ defined by (\ref{eigenapp}) of the dominant eigenvalue in Tables 1 and 2.

From Tables 1 and 2, we see that the sequence generated by Algorithm 5.1 converges to the dominant eigenvalue of the essentially nonnegative tensor without irreducibility. Algorithm 5.1 is promising for calculating the dominant eigenvalues of the test three examples.
\begin{table}[htb]
    \begin{center}
  \caption{{\small Detailed output of Algorithm 5.1 for Example \ref{exam1}} }
    \begin{tabular}
      { c  c  c  c   c  c }
      \hline
      $k$  & $\underline{\lambda}^{(k)}$ & $\overline{\lambda}^{(k)}$ &  $\lambda^{(k)}$ & $\overline{\lambda}^{(k)}-\underline{\lambda}^{(k)}$& $\Delta^{(k)}$ \\
      \hline
      1 & 35.9969 & 36.5635& 36.2802& 0.5666&  0.2833\\
       2& 36.2554&36.3030& 36.2792&0.0476& 0.0211\\
 3& 36.2747& 36.2776& 36.2762& 0.0030& 0.0015\\
 4&36.2757& 36.2758&36.2757&9.1725e-005& 4.5870e-005\\
 5& 36.2757&36.2757&36.2757&6.7568e-006& 2.9868e-006\\
 6& 36.2757& 36.2757&36.2757&4.6425e-007& 2.2441e-007\\
 7& 36.2757&36.2757&36.2757&1.9041e-008& 1.4348e-008\\
 8&36.2757&36.2757&36.2757& 8.8998e-010&8.1036e-009\\
 \hline
 \end{tabular}
  \end{center}
  \end{table}
\begin{table}[htb]
    \begin{center}
  \caption{{\small Output of Algorithm 5.1 for  Examples \ref{exam1}, \ref{exam2}, and \ref{exam3} } }
    \begin{tabular}
      { c  c  c  c  c c c c }
      \hline
      Example  &No.Iter& CPU(sec)& $\underline{\lambda}^{(k)}$ & $\overline{\lambda}^{(k)}$ &  $\lambda^{(k)}$ & $\overline{\lambda}^{(k)}-\underline{\lambda}^{(k)}$& $\Delta^{(k)}$ \\
      \hline
      \ref{exam1} &  8& 0.013&36.2757&36.2757&36.2757& 8.8998e-010&8.1036e-009\\
      \ref{exam2}&31 & 0.035& 1.0000& 1.0000& 1.0000 &9.6831e-010& 4.1210e-009\\
\ref{exam3} &37&0.078&0.8225& 0.8225&0.8225& 7.3324e-010& 1.0635e-008\\
 \hline
 \end{tabular}
  \end{center}
  \end{table}

\Section{Conclusions} In this paper, we have introduced the concepts
of essentially nonnegative tensors, which is closely related to
nonnegative tensors. The main contribution is the convexity and log
convexity of the dominant eigenvalue of an essentially nonnegative
tensor, and hence the same for the spectral radius of a nonnegative
tensor.  By modifying the Ng-Qi-Zhou algorithm \cite{s12}, we have
proposed an algorithm (Algorithm 5.1) for calculating the dominant
eigenvalue. Its convergence can be established for any essentially
nonnegative tensor without the assumptions of irreducibility and
primitiveness. Numerical results indicate that Algorithm 5.1 is
promising.

{\bf Acknowledgements.}\, Liping Zhang's  work was supported by the
National Natural Science Foundation of China(Grant No. 10871113).  Liqun Qi's work was supported by
the Hong Kong Research Grant Council.
Ziyan Luo's work was supported by the National Basic Research
Program of China (2010CB732501).

\end{document}